\documentclass[10pt, a4paper]{amsart}

\usepackage{mathrsfs}
\usepackage{tikz}
\usepackage{amssymb}
\usepackage{latexsym}
\usepackage{amsfonts}
\usepackage{amsmath}
\usepackage{eucal}
\usepackage{bm}
\usepackage{bbm}
\usepackage{graphicx}
\usepackage[english]{varioref}
\usepackage[nice]{nicefrac}
\usepackage[all]{xy}
\usepackage{amsthm}
\usepackage[margin=1.75in]{geometry}
\usepackage{amsmath,amscd}

\newcommand{\diag}{\text {\rm diag}}

\newcommand{\Hom}{{\rm Hom}}

\def\i{^{-1}}
\def\ge{\geqslant}
\def\le{\leqslant}
\def\<{\langle}
\def\>{\rangle}

\def\bb{\bold{b}}

\def\a{\alpha}

\def\d{\delta}
\def\D{\Delta}

\def\e{\epsilon}

\def\o{\omega}

\def\s{\sigma}
\def\t{\tau}

\def\l{\lambda}

\def\ZZ{\mathbb Z}

\def\AA{\mathbb A}
\def\GG{\mathbb G}
\def\NN{\mathbb N}
\def\QQ{\mathbb Q}

\def\FF{\mathbb F}
\def\RR{\mathbb R}

\def\PP{\mathbb P}

\def\co{\mathcal O}

\def\car{\mathcal R}

\def\car{\mathcal R}

\def\tW{\tilde W}
\def\tw{{\tilde w}}

\def\fs{\mathfrak S}

\def\GL{{\rm GL}}

\def\der{{\rm der}}

\def\bmu{{\mu_\bullet}}

\def\bl{{\lambda_\bullet}}
\def\bbl{\lambda_\bullet}
\def\bb{{b_\bullet}}

\def\be{{e_\bullet}}
\def\bv{v_\bullet}
\def\btw{{\tw_\bullet}}
\def\bw{{w_\bullet}}
\def\bt{{\tau_\bullet}}
\def\bg{g_\bullet}
\def\bchi{{\chi_\bullet}}
\def\bs{{\sigma_\bullet}}
\def\bbeta{{\eta_\bullet}}

\def\hbl{{\hat{\lambda}_\bullet}}
\def\hl{\hat{\lambda}}
\def\he{\hat{\eta}}

\numberwithin{equation}{subsection}
\theoremstyle{plain}
\newtheorem{thm}{Theorem}[section]

\newtheorem*{thm*}{Theorem}
 \newtheorem*{cor*}{Corollary}
 \newtheorem{prop}[thm]{Proposition}
 \newtheorem{lem}[thm]{Lemma}
 \newtheorem{cor}[thm]{Corollary}
 
 \newtheorem{rmk}[thm]{Remark}
\newtheorem{conjecture}[thm]{Conjecture}
\theoremstyle{definition}

\theoremstyle{remark}

\newtheorem*{claim*}{Claim}

\begin{document}

\title{Connectedness of Kisin varieties associated to  absolutely irreducible Galois representations}
\author{Miaofen Chen and Sian Nie}
\date{}
\address{Department of Mathematics\\
	Shanghai Key Laboratory of PMMP\\
	East China Normal University\\
	No. 500, Dong Chuan Road\\
	Shanghai 200241, China}\email{mfchen@math.ecnu.edu.cn}

\address{Academy of Mathematics and Systems Science, Chinese Academy of Sciences\\
	No. 55, Zhongguancun East Road\\
	Beijing 100190, China}\email{niesian@amss.ac.cn}


\maketitle

\begin{abstract}
We consider the Kisin variety associated to a $n$-dimensional absolutely irreducible mod $p$ Galois representation $\bar\rho$ of a $p$-adic field $K$ and a cocharacter $\mu$. Kisin conjectured that the Kisin variety is connected in this case. We show that Kisin's conjecture holds if $K$  is totally ramfied with $n=3$ or $\mu$ is of a very particular form. As an application, we also get a connectedness result for the deformation ring associated to $\bar\rho$ of given Hodge-Tate weights. We also give counterexamples to show Kisin's conjecture does not hold in general.
\end{abstract}

\section{Introduction}
Let $K$ be a finite extension of $\mathbb{Q}_p$ with $p>2$. Let $\bar\rho: \Gamma_K\rightarrow \mathrm{GL}_n(\FF)$ be a n-dimensional continuous representation of the absolute Galois group $\Gamma_K$ over a finite field $\FF$ for any $n\in\NN$. Kisin constructed in \cite{K09} a projective scheme $C_{\mu}(\bar\rho)$ over $\FF$ which parametrizes the finite flat group schemes over $\mathcal{O}_K$ with generic fiber $\bar\rho$ satisfying some determinant condition determined by $\mu$. These varieties were later called \emph{Kisin varieties} by Pappas and Rapoport (\cite{PR}). The set of connected components $\pi_0(C_{\mu}(\bar\rho))$ of the Kisin variety $C_{\mu}(\bar\rho)$ is of particular interest among the other topological invariants. Kisin showed in \cite{K09} that $\pi_0(C_{\mu}(\bar\rho))$ is in bijection with the set of connected components of the generic fiber of the flat deformation ring of $\bar\rho$ with condition on Hodge-Tate weights related to $\mu$. Moreover, Kisin determined $\pi_0(C_{\mu}(\bar\rho))$ when $K$ is totally ramified over $\QQ_p$ with $\bar\rho$ two dimensional and gave application to modularity lifting theorem.

Kisin also conjectured for  $\pi_0(C_{\mu}(\bar\rho))$ when $\bar\rho$ is indecomposable (\cite{K09} 2.4.16). When $n=2$, the conjecture says precisely that the Kisin variety has at most two connected components and this conjecture has been solved by Gee \cite{G}, Hellmann \cite{H}, Imai \cite{I} and Kisin \cite{K09}. For general $n$, we know that $C_{\mu}(\bar\rho)$ consists at most 1 point if the ramification index $e(K/\QQ_p)<p-1$ by Raynaud's result (\cite{R},3.3.2). Besides this result, very little is known about the set of connected components. Recently, Erickson and Levin studied in \cite{EL} Harder-Narasimhan theory for Kisin modules and defined a stratification on Kisin varieties by Harder-Narasimhan polygons.

In this paper, we want to study the connectedness of Kisin varieties $C_{\mu}(\bar\rho)$ for an absolutely irreducible reprensentation $\bar\rho$. In this case, Kisin's conjecture says:

\begin{conjecture}[Kisin] If $\bar\rho$ is absolutely irreducible, then $C_{\mu}(\bar\rho)$ is connected.
\end{conjecture}

We will prove the  conjecture when $\mu$ is of a particular form or $n=3$ with $K$ totally ramified. We also give counterexamples to show Kisin's conjecture does not hold in general.

We first reformulate Kisin variety in a group theoretic way.

Let $k$ be the residue field of $K$. Let $\pi$ be a uniformizer of $K$ and let $\pi_n:=\pi^{\frac{1}{p^n}}$ be a compatible system of $p^n$-th root of $\pi$ for all $n\in\mathbb{N}$. Denote $K_{\infty}:=\cup_{n} K(\pi_n)$. Then $\bar\rho_{|\Gamma_{K_{\infty}}}$ is still absolutely irreducible (\cite{B} 3.4.3).   By \cite{W}, the absolute Galois group $\Gamma_{K_{\infty}}$ is canonically isomorphic to the absolute Galois group $\Gamma_{k((u))}$ of the field of Laurent series. Hence the $\mathbb{F}$-representations of $\Gamma_{K_{\infty}}$ can be described in terms of \'etale $\varphi$-modules over $k\otimes_{\mathbb{F}_p}\mathbb{F}((u))$ (cf.\cite{FO}), where the Frobenius $\varphi$ acts on $k\otimes_{\mathbb{F}_p}\mathbb{F}((u))$ as identity on $\mathbb{F}$ and as $p$-power map on $k((u))$.
Let $(N_{\bar\rho}, \Phi_{\bar\rho})$ be the $\varphi$-module of rank $n$ over $k\otimes_{\FF_p}\FF((u))$ associated to the Tate-twist $\bar\rho(-1)_{|\Gamma_{K_{\infty}}}$.  Then
\[(N_{\bar\rho}, \Phi_{\bar\rho})\simeq ((k\otimes_{\FF_p}\FF((u)))^n, b\varphi)\] for some $b\in G(\FF((u)))$ with $G=\mathrm{Res}_{k|\FF_p}(\mathrm{GL}_n)$. In the following, we will write  $C_{\mu}(b)$ for the Kisin variety $C_{\mu}(\bar\rho)$.

Let $T$ be a maximal torus of $G$. Let $Y^+$ denotes the set of dominant cocharacters of $T$ with respect to a fixed Borel subgroup of $G$ containing $T$. It's a partially ordered set with Bruhat order denoted by $\leq$.

 Let $\bar\FF_p$ be an algebraic closure of the finite field $\FF_p$ and  $L:=\bar\FF_p((u))$ be the field of Laurent series with coefficient in $\bar\FF_p$. By Breuil-Kisin classification \cite{K06}, the Kisin variety can be described as follows:

\[C_{\mu}(b)(\bar\FF_p)=\{g\in G(L)/G(\mathcal{O}_L)| g^{-1}b\sigma (g)\in \cup_{\nu\in Y^+\atop
\nu\leq \mu} G(\mathcal{O}_L)u^{\nu}G(\mathcal{O}_L) \}\]
where the determinant condition $\mu$ can be interpreted as a dominant cocharacter $\mu\in Y^+$ and $\sigma: G(L)\rightarrow G(L)$ is induced from $\varphi$. So the Kisin variety $C_{\mu}(b)\otimes \bar\FF_p$ can be seen as a closed subvariety in the affine Grassmannian $\mathcal{F}_G=G(L)/G(\mathcal{O}_L)$ which is an ind-scheme over $\bar\FF_p$.
The isomorphism class of $C_{\mu}(b)\otimes \bar\FF_p$ depends on the $\sigma$-conjugacy class of $b$ in  $G(L)$.

The group theoretic definition of Kisin varieties resembles that of affine Deligne-Lusztig varieties which are moduli spaces of $p$-divisible groups with additional structures. Their only difference is the definition of Frobenius $\sigma$. For affine Deligne-Lusztig varieties, it's an isomorphism. But it's no longer the case for Kisin varieties. This change makes Kisin varieties much harder to study compared to affine Deligne-Lusztig varieties. Much is known about the structure of affine Deligne-Lusztig varieties by the study of many people, such as the non-emptyness (\cite{G}, \cite{GHN},\cite{He}, \cite{KR},\cite{Lu},\cite{RR}), dimension formula (\cite{GHKR}, \cite{Ha}, \cite{V06}, \cite{Z}), set of connected components (\cite{CKV}, \cite{CN19}, \cite{CN20}, \cite{HZ}, \cite{N18}, \cite{V08}) and set of irreducible components up to group action (\cite{HV}, \cite{N}, \cite{XZ},\cite{ZZ}). One of the powerful tools to study affine Deligne-Lusztig varieties is the semi-module stratification which arises in a group theoretic way. For example, de Jong and Oort used this stratification to show that each connected component of superbasic affine Deligne-Lusztig variety for $\GL_n$ is irreducible (\cite{dJO}). The second author used this stratification to give a proof of a conjecture of Xinwen Zhu and the first author about the irreducible components of affine Deligne-Lusztig varieties (\cite{N}, with another proof using twisted orbital integrals given by Zhou and Zhu \cite{ZZ}).  In this paper, we want to apply this tool to the study of Kisin varieties. The reason we restrict ourselves to $\bar\rho$ absolutely irreducible case is that we want $b$ to be simple (cf. \S \ref{section_classification simple modules}) so that we can use Caruso's classification of simple $\varphi$-modules (\cite{C}) to get a good representative $b$ in its $\sigma$-conjugacy class (Proposition \ref{simple}).

Note that $G_{|\bar\FF_p}\simeq\prod_{\tau\in \mathrm{Hom}(k, \bar\FF_p)}\mathrm{GL}_n$. An element $\mu\in Y^+$ can be written in the form $\mu=(\mu_{\tau})_{\tau}$ with \[\mu_{\tau}=(\mu_{\tau, 1},\cdots, \mu_{\tau,n})\in\ZZ^{n,+}:=\{(a_1,\cdots,a_n)\in\ZZ^n| a_1\geq a_2\geq\cdots\geq a_n\}.\]
\begin{thm*}[\ref{thm A}, \ref{thm B}]Suppose $b$ is simple and the Kisin variety $C_{\mu}(b)$ is non-empty, then $C_{\mu}(b)$ is geometrically connected if one of the following two conditions are satisfied:
\begin{enumerate}
\item $\mu=(\mu_{\tau})_{\tau}$ with $\mu_{\tau,1}\geq\mu_{\tau, 2}=\mu_{\tau, 3}=\cdots=\mu_{\tau, n}$ for all $\tau\in\mathrm{Hom}(k,\bar\FF_p)$;
\item $K$ is totally ramified and $n=3$ (i.e. $G=\GL_3$).
\end{enumerate}
\end{thm*}

The first part of the theorem recovers the main result of \cite{H} when $n=2$.

We also give counterexamples to show that Kisin's conjecture does not hold in general (cf. \S \ref{counterexmaples}).

\bigskip
 As an application of this result, we obtain a connectedness result of deformation ring of $\bar\rho$ with conditions on Hodge-Tate weights. Suppose $\bar\rho: \Gamma_K\rightarrow \mathrm{GL}_n(\FF)$ is absolutely irreducible and flat.  Let $R^{fl}$ be the flat deformation ring of $\bar\rho$ in the sense of Ramakrishna (\cite{Ra}).
Consider a minuscule cocharacter \[\nu: \GG_{m, \bar{\QQ}_p}\rightarrow (\mathrm{Res}_{K|\QQ_p}T_n)_{\bar{\QQ}_p},\] where $T_n$ is the maximal torus of $\GL_n$ consisting of diagonal matrices.
Write $R^{fl, \nu}$ be the quotient of $R^{fl}$ corresponding to deformations of Hodge-Tate weights given by $\nu$.  To $\nu$, we can associate a cocharacter \[\mu(\nu): \GG_{m,\bar\FF_p}\rightarrow \mathrm{Res}_{k|\FF_p}(T_n)_{\bar\FF_p}.\]

\begin{cor*}[\ref{coro_application_deformation}]
The scheme $\mathrm{Spec}(R^{\mathrm{fl},\nu}[\frac{1}{p}])$ is connected if one of the following two conditions holds:
\begin{enumerate}
\item $\mu(\nu)=(\mu(\nu)_{\tau})_{\tau}$ with $\mu(\nu)_{\tau}=(\mu(\nu)_{\tau,1},\cdots,\mu(\nu)_{\tau,n})$ such that $\mu(\nu)_{\tau,1}\geq\mu(\nu)_{\tau,2}=\cdots=\mu(\nu)_{\tau,n}$ for all $\tau\in\mathrm{Hom}(k,\bar\FF_p)$;
\item $K$ is totally ramified and $n=3$.
\end{enumerate}
\end{cor*}

We now a give brief outline of the article. In Section \ref{section_semimodule}, we introduce the semi-module stratification which is the main tool for us to study the geometry of Kisin varieties. In Section \ref{section_classification simple modules}, by using Caruso's classification of simple $\varphi$-modules, we get a good representative $b$ in its $\sigma$-conjugacy class for simple elements. In section \ref{section_main}, we prove the main theorems about the connectedness of Kisin varieties and also give some counterexamples for Kisin's conjecture. In section \ref{section_application}, we apply the main theorem to get a connectedness result for deformation rings.

\bigskip
\textbf{Acknowledgments.}  We would like to thank Brandon Levin who encouraged us to work on this topic. We thank Hui Gao for useful discussions. We also thank Mark Kisin and Frank Calegari for their useful comments. The first author was partially supported by NSFC grant No.11671136 and STCSM grant No.18dz2271000. The second author was partially supported by NSFC (Nos.11922119, 11688101 and 11621061) and QYZDB-SSW-SYS007.

\section{Semi-module stratification on Kisin varieties}\label{section_semimodule}
In this section, suppose $p$ is a prime number (we allowed $p=2$).
\subsection{Kisin variety for arbitary reductive groups}
Recall that $L = \bar\FF_p ((u))$ and $\co_L = \bar\FF_p[[u]]$. Let $\varphi_L: L \to L$ be the homomorphism given by $\sum_i a_i u^i \mapsto \sum_i a_i u^{p i}$, where $a_i \in \bar\FF_p$.

Let $G$ be a reductive group over $\FF_p$. Fix $T \subseteq B \subseteq G$ with $T$ a maximal torus and $B$ a Borel subgroup of $G$. Let $\car = (X, \Phi, Y, \Phi^\vee)$ be the associated root datum, where $\Phi$ (resp. $\Phi^\vee$) is the set of roots (resp. coroots) of $G$; $X$ (resp. $Y$) is the character group (resp. cocharacter group) of $T$, together with a perfect pairing $\< , \>: X \times Y \to \ZZ$. Denote by $\Phi^+$ the set of positive roots appearing in $B$, and by $Y^+$ the corresponding set of dominant cocharacters. On $Y^+$ we have the Bruhat order denoted by $\leq$.

Let $G(L)$ be the group of $L$-points of $G$.  We have the Cartan decomposition on $G(L)$:

\[G(L)=\coprod_{\mu\in Y^+} G(\mathcal{O}_L)u^{\mu}G(\mathcal{O}_L).\] It's a stratification in the strict sense, i.e., for any $\mu\in Y^+$,
\[\overline{G(\mathcal{O}_L)u^{\mu}G(\mathcal{O}_L)}=\coprod_{\nu\in Y^+\atop \nu\leq \mu} G(\mathcal{O}_L)u^{\nu}G(\mathcal{O}_L).\]

Fix an automorphism $\sigma_0$ of $G$ that fixes $T$ and $B$. Denote by $\s$ the endomorphism of $G(L)$ induced from $\varphi$ and $\s_0$ as follows:
\[\s: G(L)\stackrel{\s_0}{\longrightarrow} G(L) \stackrel{\varphi_L}{\longrightarrow} G(L). \]

 Let $\FF$ be a finite extension of $\FF_p$.  Let $b \in G(\FF ((u)))$ and $\mu \in Y^+$ such that $\FF$ contains the reflex field of the conjugacy class of $\mu$.  Then the associated (closed) Kisin variety is defined over $\FF$ and with  $\bar\FF_p$-points given by $$C_\mu(b)(\bar\FF_p)  = \{g \in G(L); g\i b \s(g) \in \overline{G(\co_L) u^\mu G(\co_L)} \} / G(\co_L).$$

 In the literature, Kisin varieties were defined for the groups of the form $G=\mathrm{Res}_{k|\FF_p}H$ and the automorphism $\s_0$ is induced by the Frobenius relative to $k|\FF_p$, where $k$ is finite extension of $\FF_p$ and $H$ is a reductive group over $k$. (cf. \cite{PR}, \cite{EL}). In this case, the Kisin variety has the moduli interpretation that parametries some finite flat models with $H$-structure of a Galois representation with value in $H$ (\cite{Le}).

\subsection{Semi-module stratification of Kisin varieties}
In the rest of this section, we are interested in $G\times_{\FF_p} \bar\FF_p$ and $C_{\mu}(b)\times_{\FF}\bar\FF_p$ but not their rational structure.

Let $N_T \subseteq G$ be the normalizer of $T$. Denote by $W_0 = N_T(L) / T(L)$ the Weyl group, and denote by $$\tW = Y \rtimes W_0 = \{u^\t w; \l \in Y, w \in W_0\}$$ the Iwahori-Weyl group. There is a natural action of $\tW$ on the vector space $Y_\RR = Y \otimes \RR$. For $\tw = u^\t w \in \tW$ we set $\dot\tw = u^\t \dot w$ with $\dot w \in N_T(\bar\FF_p)$ some/any lift of $w$. As $\s$ and $\varphi_G$ preserve $T(L)$, we still denote by $\s$ and $\s_0$ the induced endomorphisms on $Y$ respectively. Notice that $\s_0$ is an automorphism of the root datum $\car$ and $\s = p \s_0$.

Let $I \subseteq G(L)$ be the Iwahori subgroup associated to the fundamental alcove $$\D = \{v \in Y_\RR; 0 < \<\a, v\> < 1, \forall \a \in \Phi^+\}.$$ Namely, $I$ is the preimage of $B^{-}(\bar\FF_p)$ under natural map $G(\mathcal{O}_L) \overset {u \mapsto 0} \longrightarrow G(\bar\FF_p)$, where $B^{-}$ is the opposite Borel subgroup of $B$.

 Denote by $I_\der = [I, I]$ the derived subgroup of $I$. For any $r\in\NN$, denote by $T(1+u^r\co_L)$ the image of $(1+u^r\co_L)^n\subset \GG_m^n(L)$ via an isomorphism of algebraic groups $\GG_m^n\simeq T_{|\bar\FF_p}$, where $n$ is the rank of $T$. It's easy to see  $T(1+u^r\co_L)$ does not depend on the choice of the isomorphism $\GG_m^n\simeq T_{|\bar\FF_p}$.
\begin{lem} \label{iota}
Let $b = \dot \tw$ with $\tw \in \tW$ such that the (unique) fixed point of $\tw\s$ lies in $\D$. Then the Lang's map $\Psi: h \mapsto h\i b\s(h) b\i$ restricts to a bijection $I_\der \cong I_\der$.
\end{lem}

\begin{proof} The first condition implies that $b\sigma (I)b^{-1}\subseteq I$ and hence $b\sigma (I_{\der})b^{-1}\subseteq I_{\der}$. It follows that $\Psi(I_\der)\subseteq I_\der$. To show $\Psi$ is a bijection, it suffices to show the action $h \mapsto b \s(h) b\i$ is topologically unipotent on $I_\der$. Let $e \in \D$ be the fixed point of $\tw\s$. For $r \in \RR_{> 0}$ we denote by $I_{\der, r}$ the Moy-Prasad subgroup generated by $T(1 + u^{\lceil r \rceil} \co_L)$ and $U_\a(u^k \co_L)$ such that $-\<\a, e\> + k \ge r$. It suffices to show $b \s(I_{\der, r}) b\i \subseteq I_{\der, p r}$. Indeed, write $\tw = u^\t w$ with $\t \in Y$ and $w \in W_0$. Then $e =w\s(e) + \t = pw\s_0(e) + \t$. One computes that \begin{align*} b \s(T(1 + u^{\lceil r \rceil}\co_L)) b\i &= T(1 + u^{p \lceil r \rceil}\co_L) \subseteq T(1 + u^{\lceil pr \rceil}\co_L); \\ b \s(U_\a(u^k \co_L)) b\i &= U_{w\s_0(\a)}(u^{pk + \<w\s_0(\a), \t\>} \co_L) \subseteq I_{\der, pr}, \end{align*} where the second inclusion follows from that \begin{align*}  -\<w\s_0(\a), e\> + pk + \<w\s_0(\a), \t\> &= -\<w\s_0(\a), pw\s_0(e) + \t\> + pk + \<w\s_0(\a), \t\> \\  &= p(-\<\a, e\> + k) \geq pr.\end{align*} This finishes the proof.
\end{proof}

For $\l \in Y$, let $U_\l^+$ (resp. $U_\l^-$) be the maximal unipotent subgroup of $G$ generated by $U_\a$ such that $\l_\a \ge 0$ (resp. $\l_\a < 0$), where
\[ \l_\a =\begin{cases} \<\l, \a\>,  &\text{ if } \a\in \Phi^-;\\ \<\l, \a\> - 1, &\text{ if } \a\in\Phi^+. \end{cases} \]
Note that $U_\l^+$ and $U_\l^-$ are opposite to each other. If there is no confusion, we also write $U_\l^+$ (resp. $U_\l^{-}$) for $U_\l^+(L)$ (resp. $U_\l^{-}(L)$).

Consider the following semi-module decomposition $$C_\mu(b)(\bar\FF_p) = \sqcup_{\l \in Y} C_\mu^\l(b)(\bar\FF_p),$$ where each piece $C_{\mu}^\l(b)$ is locally closed subscheme of $C_{\mu}(b)\times_{\FF}\bar\FF_p$ with $\bar\FF_p$-points $$C_\mu^\l(b)(\bar\FF_p) = (I u^\l G(\co_L)/G(\co_L)) \cap C_\mu(b)(\bar\FF_p).$$

\begin{prop}\label{semi-module} Let $b = \dot\tw$, where $\tw = u^\t w \in \tW$ is as in Lemma \ref{iota} for some $\t \in Y$ and $w \in W_0$. The following conditions are equivalent:
\begin{enumerate}
\item $C_\mu^\l(b)$ is non-empty;
\item $u^{\lambda}\in C_{\mu}(b)(\bar\FF_p)$;
\item $\l^\natural: =-\lambda+\t+w\sigma(\lambda)\leq \mu$.
\end{enumerate}
Moreover, under these equivalent conditions we have

(a) $C_{\mu}^\l(b)$ is connected;

(b) $C_\mu(b) = \{u^\l\}$ if and only if \[I \cap U_\l^+ \cap u^\l \overline{G(\co_L) u^\mu G(\co_L)} u^{-\l^\dag} = u^\l U_\l^+(\co_L) u^{-\l},\] where $\l^\dag = \t + w\s(\l);$

(c) $C_{\mu}^\l(b)$ is irreducible of dimension $|R(\l)|$ if $\mu$ is minuscule, where $R(\l) = \{\a \in \Phi; \l_\a \ge 1, \<\a, \l^\natural\> = -1\}$.
\end{prop}
\begin{proof}We first show the three conditions are equivalent. Note that $u^{-\l}b\sigma(u^\l)=u^{\l^\natural}$. It follows the implications (3) $\Rightarrow$ (2) $\Rightarrow$ (1). It remains to show the implication (1) $\Rightarrow$ (3). Let $\iota: I_\der \to I_\der$ be the endomorphism of $I_\der$ given by $h \mapsto b \sigma(h) b\i$. Let $\Psi: I_\der \to I_\der$ be the Lang's map given by $h \mapsto h\i \iota(h)$. By Lemma \ref{iota}, $\Psi$ is an isomorphism whose inverse is given by $h \mapsto \cdots \iota^2(h)\i  \iota(h)\i h\i$. Denote by $\pi_\l: I_\der \to Iu^\l G(\co_L)/G(\co_L)$ the surjective map given by $h \mapsto h u^\l G(\co_L)$. Then one computes that $$\pi\i_\l(C_{\mu}^\l(b)(\bar\FF_p)) = \Psi\i(I_\der \cap u^\l \overline{G(\co_L) u^\mu G(\co_L)} u^{-\l^\dag}),$$ with $\l^\dag = \t + w\sigma(\l)$. By the Iwasawa decomposition, we have $$I_\der = (I \cap U_\l^-) T(1+u\co_L) (I \cap U_\l^+).$$

Noticing that $u^{-\l} (I \cap U_\l^-) T(1+u\co_L) u^\l \subseteq I$, we deduce that \begin{eqnarray} \label{eqn_Ider} &I_\der \cap u^\l \overline{G(\co_L) u^\mu G(\co_L)} u^{-\l^\dag}
\\ \nonumber= &(I \cap U_\l^-) T(1+u\co_L) \ ((I \cap U_\l^+) \cap u^\l \overline{G(\co_L) u^\mu G(\co_L)} u^{-\l^\dag}).\end{eqnarray} Suppose $C_{\mu}^\l(b)(\bar\FF_p)$ is non-empty, that is, $I_\der \cap u^\l \overline{G(\co_L) u^\mu G(\co_L)} u^{-\l^\dag} \neq \emptyset$. Thus $$\emptyset \neq (u^{-\l} (I \cap U_\l^+) u^{\l^\dag}) \cap \overline{G(\co_L) u^\mu G(\co_L)}  \subseteq U_\l^+ u^{\l^\natural} \cap \overline{G(\co_L) u^\mu G(\co_L)},$$ which means that $\l^\natural \leq \mu$ and the implication (1) $\Rightarrow$ (3) is proved.

\

Now suppose that the three equivalent conditions are satisfied. Choose $n \gg 0$ such that both $I_\der \cap u^\l \overline{G(\co_L) u^\mu G(\co_L)} u^{-\l^\dag}$ and $I_\der \cap u^\l G(\co_L) u^{-\l}$ are invariant under the left multiplication by $I_{\der, n}$, see the proof of Lemma \ref{iota}. Then $\Psi$ induces an automorphism $$\Psi_n: I_\der / I_{\der, n} \overset \sim \to I_\der / I_{\der, n}.$$ Now we have \begin{align*}C_{\mu}^\l(b)(\bar\FF_p) &= \pi_\l \circ \Psi\i (I_\der \cap u^\l \overline{G(\co_L) u^\mu G(\co_L)} u^{-\l^\dag}) \\ &\cong \Psi\i(I_\der \cap u^\l \overline{G(\co_L) u^\mu G(\co_L)} u^{-\l^\dag}) / I_\der \cap u^\l G(\co_L) u^{-\l} \\ &\cong \Psi_n\i(I_\der \cap u^\l \overline{G(\co_L) u^\mu G(\co_L)} u^{-\l^\dag} / I_{\der, n}) /  ((I_\der \cap u^\l G(\co_L) u^{-\l}) / I_{\der, n}). \end{align*}

As $\Psi_n$ is an isomorphism, to show $C_\mu^\l(b)$ is connected, it suffices to show $$I_\der \cap u^\l \overline{G(\co_L) u^\mu G(\co_L)} u^{-\l^\dag} / I_{\der, n}$$ is connected. In view of (\ref{eqn_Ider}) and that $(I \cap U_\l^-) T(1+u\co_L)$ is connected, this is equivalent to show $$(I \cap U_\l^+) \cap u^\l \overline{G(\co_L) u^\mu G(\co_L)} u^{-\l^\dag} / I_{\der, n} \cap U_\l^+$$ is connected. Let $\xi$ be a regular dominant cocharacter with respect to $U_\l^+$. Then the map $z \mapsto z^\xi h z^{-\xi}$ defines an affine line connecting an arbitrary point $h \in (I \cap U_\l^+) \cap u^\l \overline{G(\co_L) u^\mu G(\co_L)} u^{-\l^\dag} / I_{\der, n} \cap U_\l^+$ and the identity $1$. This finishes the proof of (a).

\

One computes that \begin{align*}\dim C_\mu^\l(b) &= \dim (I_\der \cap u^\l \overline{G(\co_L) u^\mu G(\co_L)} u^{-\l^\dag} )/ I_{\der, n} - \dim (I_\der \cap u^\l G(\co_L) u^{-\l}) / I_{\der, n} \\ &=\dim (I\cap U_\l^-) T(1 + u\co_L) ((I \cap U_\l^+) \cap u^\l \overline{G(\co_L) u^\mu G(\co_L)} u^{-\l^\dag}) / I_{\der, n} \\ &\quad - \dim (I\cap U_\l^-) T(1 + u\co_L) (u^\l U_\l^+(\co_L) u^{-\l}) / I_{\der, n} \\ &= \dim (I \cap U_\l^+) \cap u^\l \overline{G(\co_L) u^\mu G(\co_L)} u^{-\l^\dag}) / I_{\der, n} \cap U_\l^+ \\ &\quad - \dim (u^\l U_\l^+(\co_L) u^{-\l}) / I_{\der, n} \cap U_\l^+, \end{align*} where the second equality follows from (\ref{eqn_Ider}) and  \[I_\der \cap u^\l G(\co_L) u^{-\l} = (I\cap U_\l^-) T(1 + u\co_L) (u^\l U_\l^+(\co_L) u^{-\l}).\] By (a), $C_\mu^\l(b) = \{u^\l\}$ if and only if $\dim C_\mu^\l(b) = 0$, that is, \[I \cap U_\l^+ \cap u^\l \overline{G(\co_L) u^\mu G(\co_L)} u^{-\l^\dag} = u^\l U_\l^+(\co_L) u^{-\l}\] since $$u^\l U_\l^+(\co_L) u^{-\l} \subseteq I \cap U_\l^+ \cap u^\l \overline{G(\co_L) u^\mu G(\co_L)} u^{-\l^\dag}$$ with both sides connected and are invariant under left multiplication by $u^\l U_\l^+(\co_L) u^{-\l}$. Thus (b) is proved.

\

Suppose $\mu$ is minuscule, then $\l^\natural$ is conjugate to $\mu$ by (3). In particular, $$U_\l^+ u^{\lambda^\natural} \cap G(\co_L) u^\mu G(\co_L) = U_\l^+(\co_L) u^{\lambda^\natural} U_\l^+(\co_L) = U_\l^+(\co_L) (\prod_{\a \in D} U_\a(t\i \co_L)) u^{\lambda^\natural},$$ where $D = \{\a \in \Phi; \l_\a \ge 0, \<\a, \l^\natural\> = -1\}$. As $\l^\natural$ is minuscule, all the root subgroups $U_\a$ for $\a \in D$ commute with each other. Therefore, \begin{align*} \tag{*} &\quad\ (I \cap U_\l^+) \cap u^\l \overline{G(\co_L) u^\mu G(\co_L)} u^{-\l^\dag} \\ &=(I \cap U_\l^+) \cap u^\l G(\co_L) u^\mu G(\co_L) u^{-\l^\dag} \\ &= u^\l (u^{-\l} (I \cap U_\l^+) u^\l u^{\lambda^\natural} \cap G(\co_L) u^\mu G(\co_L)) u^{-\l^\dag} \\  &= u^\l ((u^{-\l} (I \cap U_\l^+) u^\l u^{\lambda^\natural}) \cap (U_\l^+(\co_L) (\prod_{\a \in D} U_\a(t\i \co_L)) u^{\lambda^\natural})) u^{-\l^\dag} \\ &= I \cap ((u^\l U_\l^+(\co_L) u^{-\l}) \prod_{\a \in D} U_\a(u^{\<\a, \l\> - 1} \co_L) \\ &= (u^\l U_\l^+(\co_L) u^{-\l}) (I \cap  (\prod_{\a \in D} U_\a(u^{\<\a, \l\> - 1} \co_L)) \\ &= (u^\l U_\l^+(\co_L) u^{-\l}) \prod_{\a \in D} I \cap U_\a(u^{\<\a, \l\> - 1} \co_L) \\ &=  (u^\l U_\l^+(\co_L) u^{-\l}) \prod_{\a \in D \setminus R(\l)} U_\a(u^{\<\a, \l\>} \co_L) \prod_{\a \in R(\l)} U_\a(u^{\<\a, \l\> - 1} \co_L) \\ &=(u^\l U_\l^+(\co_L) u^{-\l})  \prod_{\a \in R(\l)} U_\a(u^{\<\a, \l\> - 1} \overline \FF_p), \end{align*}where the fifth equality follows from that $u^\l U_\l^+(\co_L) u^{-\l} \subseteq I$; the sixth one follows from that $I \cap U_\l^+ = \prod_{\a \in \Phi, \l_\a \ge 0}(I \cap U_\a)$; the last but one equality follows from that \[I \cap U_\a(u^{\<\l, \a\> - 1}\co_L)=\begin{cases} U_\a(u^{\<\l, \a\> - 1}\co_L), &\text{ if }\l_\a \ge 1;\\ U_\a(u^{\<\l, \a\>}\co_L), &\text{ if }\l_\a = 0.\end{cases}\] In view of \ref{eqn_Ider} and (*), $I_\der \cap u^\l \overline{G(\co_L) u^\mu G(\co_L)} u^{-\l^\dag} / I_{\der, n}$ is irreducible. Hence $C_\mu^\l(b)$ is irreducible with dimension \begin{align*} \dim C_\mu^\l(b) &=  \dim (I \cap U_\l^+) \cap u^\l \overline{G(\co_L) u^\mu G(\co_L)} u^{-\l^\dag}) / I_{\der, n} \cap U_\l^+ \\ &\quad - \dim (u^\l U_\l^+(\co_L) u^{-\l}) / I_{\der, n} \cap U_\l^+ \\ &= |R(\l)|.\end{align*} This finishes the proof of (c).
\end{proof}

\subsection{Multi-copy case}\label{subsection_muticopy}

Let $G^d$ be the product of $d$ copies of $G$. For $b \in G(L)$ and $\bmu \in Y^d$ we define $C_{\bmu}(b)$ over $\bar\FF_p$ with $\bar\FF_p$-points: $$C_{\bmu}(b)(\bar\FF_p) = \{\bg \in G^d(L); \bg\i \bb \bs(\bg) \in \overline{G^d(\co_L) u^\bmu G^d(\co_L)}\} / G^d(\co_L),$$ where $\bb = (1, \dots, 1, b) \in G^d(L)$ and $\bs: G^d(L) \to G^d(L)$ is the endomorphism given by $(g_1, g_2, \dots, g_d) \mapsto (g_2, \dots, g_d, \s(g_1))$. We still denote by $\bs$ the induced linear map $Y_\RR^d \to Y_\RR^d$ given by $(v_1, v_2, \dots, v_d) \mapsto (v_2, \dots, v_d, \s(v_1))$.
\begin{lem}\label{lemma_multi}
Let $\bmu = (\mu_1, \dots, \mu_d)$ be a dominant cocharacter and let $\mu = \mu_1 + \cdots + \mu_d$. Then the projection $G^d(L) \to G(L)$ to the first factor induces a surjection $C_{\bmu}(b)(\bar\FF_p) \twoheadrightarrow C_\mu(b)(\bar\FF_p)$.
\end{lem}
\begin{proof}Note that
\[\overline{G(\co_L)u^{\mu_1}G(\co_L)}\cdot \overline{G(\co_L)u^{\mu_2}G(\co_L)}\cdots \overline{G(\co_L)u^{\mu_d}G(\co_L)}=\overline{G(\co_L)u^{\mu}G(\co_L)} \] by \cite{BT} 4.4.4. The lemma follows by direct computation.
\end{proof}

Suppose $b = \dot \tw$ for some $\tw = u^\t w \in \tW$. Set $\btw = (1, \dots, 1, \tw) = u^{\bt} \bw$ with $\bt = (0, \dots, 0, \t)$ and $\bw = (1, \dots, 1, w)$. Let $\bl \in Y^d$. Define \begin{gather*}\l_\bullet^\natural = -\bl + \bt + \bw \bs(\bl); \\ R(\bl) = \coprod_{i=1}^d\{\a \in \Phi; (\l_i)_\a \ge 1, \<\a, (\l_\bullet^\natural)_i\> = -1\}.\end{gather*} Moreover, we set $C_\bmu^\bl(\bb)(\bar\FF_p)=  I^d u^\bl G^d(\co_L) / G^d(\co_L) \cap C_\bmu(b)(\bar\FF_p)$.
\begin{prop} \label{dim}
Suppose $\bmu$ is minuscule and $b$ is as in Lemma \ref{iota}. Then $C_\bmu^\bl(b) \neq \emptyset$ if and only if $\l_\bullet^\natural$ is conjugate to $\bmu$. Moreover, in this case, $C_\bmu^\bl(\bb)$ is irreducible of dimension $|R(\bl)|$.
\end{prop}
\begin{proof}
By Lemma \ref{iota}, the endomorphism of $I_\der^d$ given by $g_\bullet \mapsto \bb \bs(g_\bullet) \bb\i$ is topologically unipotent. Hence the Lang's map $I_\der^d \to I_\der^d$ given by $g_\bullet \mapsto g_\bullet\i \bb \bs(g_\bullet) \bb\i$ is a bijection. Moreover, the fixed point of $\btw \bs$ lies in the fundamental alcove of  $Y_\RR^d$. Then the statement follows the same way as Proposition \ref{semi-module} (c).
\end{proof}

\section{Classification of simple $\varphi$-modules}\label{section_classification simple modules}
Suppose $G=\mathrm{Res}_{k|\FF_p}H$ with $k=\FF_q$ and $H$ is a reductive group over $k$, equipped with $\s_0$ induced from the Frobenius relative to $k|\FF_p$. We say $b \in G(L)$ is {\it simple} if $b \notin P(L)$ for any $\s$-stable proper parabolic subgroup $P$ of $G$ up to $\s$-conjugation. It's equivalent to say the $\varphi$-module
$$((k\otimes_{\FF_p}\bar\FF_p((u)))^n, b\varphi)$$ over $k\otimes\bar\FF_p((u))$ is irreducible, where $\varphi$ is as in the introduction.
Let $B(G)$ be the set of $\sigma$-conjugacy classes in $G(L)$. A $\sigma$-conjugacy class $[b]\in B(G)$ is called simple if all/some representative in this $\sigma$-conjugacy class is simple. Let $B(G)_s$ be the subset of $B(G)$ consisting of simple elements. Notice that $C_\mu(b)\times_{\FF}\bar\FF_p$ depends on the $[b]\in B(G)$. The goal of this section is to give a good representative of $b$ in its $\sigma$-conjugacy class.

There is a natural identification $G(L)=\prod_{i=1}^f H(L)$ with $f=[k:\FF_p]$. Under this identification, we have
\[\begin{split} \sigma: G(L)&\longrightarrow G(L)\\ (x_1,\cdots, x_f)&\longmapsto (\varphi_L(x_2), \cdots, \varphi_L(x_f), \varphi_L(x_1)).\end{split}\]

For the group $H$ over $k$, we define $T_H\subset B_H$ maximal torus and Borel subgroup of $H$ such that $T=\mathrm{Res}_{k|\FF_p}T_H$ and $B=\mathrm{Res}_{k|\FF_p}B_H$. We also define $W_{0,H}$, $X_H$, $Y_H$, $\Delta_H$, $\Phi_H$, $\Phi_H^+$ in the same way for $H$. Then there is natural identifications
\[Y\simeq \prod_{i=1}^f Y_H \text{ and } X\simeq \prod_{i=1}^f X_H.\]The endomorphism $\s: Y \to Y$ is given by $(h_1, h_2, \dots, h_f) \mapsto (ph_2, \dots, ph_f, ph_1)$.

Using the Frobenius relative to $\bar\FF_p|k$  instead of the Frobenius relative to $\bar\FF_p|\FF_p$, we define similarly $\sigma_H: H(L)\rightarrow H(L)$ and simple elements in $H(L)$. We also define $B(H)$ and $B(H)_s$ in the same way.

We can easily check the following lemma.
\begin{lem}\label{lemma_Shapiro}There is a Shapiro bijection
\[\begin{split} B(H)&\simeq B(G)\\ [b']&\mapsto [b],\end{split}\]where $b=(b', 1,\cdots, 1)\in G(L)=\prod_{i=1}^f H(L)$. Moreover, it induces a bijection on the subset of simple elements $B(H)_s\simeq B(G)_s$.
\end{lem}

From now on, suppose $H=\GL_n$.  Let $T_H$ (resp. $B_H$) be the subgroup of diagonal (resp. upper triangular) matrices. We have the following identifications: \begin{itemize}
\item $X_H=\oplus_{i=1}^n \ZZ e_i\simeq \ZZ^n$, $Y_H= \oplus_{i=1}^n \ZZ e_i^\vee\simeq \ZZ^n$ with the pairing  $X_H \times Y_H \to \ZZ$ induced by $\<e_i, e_j^\vee\> = \d_{i, j}$;
\item $\Phi = \sqcup_{i=1}^f \Phi_H$, where $\Phi_n = \{\a_{i, j} = e_i - e_j; 1 \le i \neq j \le n\}$;
\item $\Phi^+ = \sqcup_{i=1}^f \Phi_H^+$, where  $\Phi_H^+ = \{\a_{i, j}; 1 \le i < j \le n\}$;
\item $W_0 =W_{0,H}^f$ with $W_{0,H}=\fs_n$;
\item $\Delta=\Delta_H^f$, with $\Delta_H=\{(x_1,\cdots, x_n)\in\RR^n| x_1>x_2>\cdots>x_n>x_1-1\}$.
\end{itemize}

We will need the following result of Caruso.

\begin{prop}[Caruso]\label{Caruso}Suppose $[b]\in B(H)$ is simple. Then there exists a representative $b=au^{\tau'}w'\in [b]$ where \begin{enumerate}
\item $a\in H(\bar\FF_p)$ is central,
\item $w'\in W_{0,H}=\fs_n$ is the cyclic permutation given by $i\mapsto i+1 \mod n$ for $1\leq i\leq n$;
\item $\tau'=(m,0,\cdots,0)\in Y_H=\ZZ^n$ such that $\frac{m (q^{n'}-1)}{q^n-1}\notin\ZZ$ for any $n'|n$ with $1\leq n'<n$.
\end{enumerate}
\end{prop}
\begin{proof}By \cite[Corollaire 8]{C}, we may assume $b=a'u^{\tau'}w'$ where $w'$ as in (2), $a'=\diag(a'_1,1,\cdots, 1)$ with $a'_1\in\bar\FF_p^*$ and $\tau'=(m,0,\cdots,0)$. Let $a=\diag(c,\cdots,c)$ with $c\in\bar\FF_p^*$ such that $c^n=a'_1$. It's easy to check $g^{-1}b\sigma_H(g)=au^{\tau'}w'$, where $g=\diag(c^{n-1},\cdots,c, 1)$. So (1) is also satisfied. The condition (3) follows from \cite[Proposition 3]{C} as $a$ is central.
\end{proof}

\begin{prop}  \label{simple}
If $b \in G(L)$ is simple, then it is $\s$-conjugate to some lift $\dot\tw$ of $\tw \in \tW$ such that the fixed point $\tw\s$ lies in $\D \cap (\RR \setminus \ZZ)^{fn} \subseteq \RR^{fn} = Y_\RR$.
\end{prop}
\begin{proof}
By Lemma \ref{lemma_Shapiro}, any simple element in B(G) is of the form $[b]$, where $b=(b',1,\cdots, 1)$ and $b'$ is a lift of $\tw' = u^{\t'} w'$, where $\t' = (m, 0, \dots, 0) \in Y_H = \ZZ^n$ and $w' \in W_{0,H}=\fs_n$ as in Proposition \ref{Caruso}. Then $b$ is a lift of $\tw = u^{\t} w$ where $\t=(\t',0,\cdots, 0)\in Y$ and $w=(w',1,\cdots, 1)\in W$. Let $e$ (resp. $e'$) be the fixed point of $\tw\sigma$ on $Y_\RR=\RR^{fn}$ (resp. $\tw'\sigma_H$ on $Y_{H,\RR}=\RR^n$).

\emph{Claim: $e\in (\RR\backslash\ZZ)^{fn}$. Moreover, $e$ lies in some alcove $\D_1$ in $Y_\RR$.}

By the description of $b'$, we can compute that $$e' = - (\frac{m}{p^n - 1}, \frac{m p}{p^n - 1}, \dots, \frac{m p^{n-1}}{p^n - 1}).$$ Moreover $$\frac{m p^i}{p^n - 1} - \frac{m p^j}{p^n - 1} \notin \ZZ \text{ for } 0 \le i < j \le n-1$$ by the condition (3) in Proposition \ref{Caruso} combined with the easy fact that
\[\mathrm{gcd}(q^{a_1}-1, q^{a_2}-1)=q^{\mathrm{gcd}(a_1,a_2)}-1\] for any positive integers $a_1$ and $a_2$. In particular, this implies $\frac{m p^{i-1}}{p^n - 1} \notin \ZZ$ (for $1 \le i \le n$) and hence $e' \in (\RR \setminus \ZZ)^n$. Note that the action of $(b\sigma)^f$ on $G(L)=\prod_{i=1}^f H(L)$ stablizes each component of $H(L)$ and its restriction to the first component equals to $b'\sigma_H$. It follows that
  \[e=(e', pe', \cdots, p^{f-1}e').\] This implies the Claim.

Let $z \in \tW$ such that $\D_1 = z(\D)$ and let $\tw = z\i \tw \s(z)$. Then the fixed point $z\i(e')$ of $\tw \s$ lies in $\D$ as desired.
\end{proof}

\section{Main results}\label{section_main}
In this section, let $G=\mathrm{Res}_{k|\FF_p}\GL_n$. Then $G_{\bar\FF_p}\simeq(\GL_n)^f$ with $f=[k: \FF_p]$. We will prove two connectedness results for the Kisin variety $C_\mu(b)$.

\subsection{} The first result is the following for $\mu$ of particular form.
\begin{thm} \label{thm A}
Suppose $b \in G(L)$ is simple and $\mu=(\mu_1,\cdots,\mu_f)\in (\ZZ^n)^f$ with $\mu_i=(\mu_{i,1}, \cdots, \mu_{i, n})$ satisfying $\mu_{i,1}\geq \mu_{i, 2}=\cdots=\mu_{i, n}$ for all $1\leq i\leq f$. Then  $C_\mu(b)$ is geometrically connected if it is nonempty.
\end{thm}

\begin{rmk} Theorem \ref{thm A} is proved in \cite[Theorem 1,1]{H} for $n=2$.
\end{rmk}

In order to prove this theorem, we need a connectedness result for Kisin varieties in the multi-copy case (cf.\S \ref{subsection_muticopy}). For any positive integer $d$, consider the group $G^d$. Let $N=df$. We fix two identifications:
\[\begin{split}Y^d &\simeq (\ZZ^n)^N\\ \bv = (v_1, \dots, v_d) &\mapsto (v^1, \cdots, v^N)\end{split} \]
where $v^{i + (j - 1) d}=v_{i, j} $ for any $1\leq i\leq d$ and $1\leq j\leq f$ with $v_i = (v_{i, 1}, \dots, v_{i, f}) \in Y = (\RR^n)^f$ and
\[\begin{split}
 W_0^d &\simeq (\fs_n)^N\\ \bw = (w_1, \dots, w_d)&\mapsto (w^1, \cdots, w^N)
\end{split}\]
where $w^{i + (j - 1) d}=w_{i, j} $ with $w_i = (w_{i, 1}, \dots, w_{i, f}) \in W_0=\fs_n^f$.

For $\bv\in Y^d$ and $\bw\in W_0^d$, we also define $v^{N+1}:=v^1$ and $w^{N+1}:=w^1$.

\bigskip
Let $\o_1^\vee = (1, 0, \dots, 0) \in \ZZ^n$. Let $\bmu \in Y^d = (\ZZ^n)^N$ such that $\mu^k = m^k \o_1^\vee$ with $m^k \in \{0, 1\}$ for $1 \le k \le N$. Let $b \in G(L)$ be as in Lemma \ref{iota}. We consider the variety $C_\bmu(b)$.
\begin{thm}\label{thm A'}
Let $\bmu$ and $b$ be as above such that $C_\bmu(b) \neq \emptyset$. Then there exists a unique $\bl \in Y^d$ such that $\dim C_\bmu^\bl(b) = 0$. Moreover, $C_\bmu(\bb)$ is connected.
\end{thm}

We first use this result to prove Theorem \ref{thm A}
\begin{proof}[Proof of Theorem \ref{thm A}] If $\chi \in Y$ is a central, then we have the identification $$C_\mu(b) = C_{\mu + \chi}(u^\chi b).$$ By replacing $\mu$ and $b$ with $\mu + \chi$ and $u^\chi b$ respectively for some suitable central cocharacter $\chi$, we may assume that $\mu = (m_1 \o_1^\vee, \dots, m_f \o_1^\vee)$ with $m_k \in \ZZ_{\ge 0}$ for $1 \le k \le f$. By Proposition \ref{simple}, we may assume further that $b$ satisfies the condition in Lemma \ref{iota}. The theorem then follows from Theorem \ref{thm A'} and Lemma \ref{lemma_multi}.
\end{proof}

Now it remains to prove Theorem \ref{thm A'}. We need some combinatorial preparations.

For any $v=(v(1),\cdots, v(n)) \in \RR^n$, let \begin{itemize}
\item $[v] = \{v(k); 1 \le k \le n\} \subseteq \RR$,
\item $\<v\> = v(1) + \cdots + v(n) \in \RR$,
\item $\d(v) = \<v\> - n \min[v] \in \RR$,
\item $h(v) = \sum_{i=1}^n \lfloor v(i) - \min[v] \rfloor \in \ZZ_{\ge 0}$.
\end{itemize}
Note that $h(v)\leq \d(v)<h(v)+n-1$. Moreover, $h(v)=0$ if and only if $\max[v] - \min[v] < 1$.

 Define $\varsigma(v) \in \RR^n$ such that
 \[\varsigma(v)(i) = \begin{cases}v(i) - 1, &\text{ if } v(i) = \max[v]\\  v(i), &\text{otherwise}.\end{cases}\]
  For $l \in \ZZ_{> 0}$, set $\varsigma^l(v) = \varsigma \circ \cdots \circ \varsigma(v)$, where $\varsigma$ appears $l$ times.

\begin{lem} \label{delta}
Let $e \in \RR^n$ such that $e(i) - e(j) \notin \ZZ$ for $1 \le i < j \le n$. Let $v, v' \in \ZZ^n - e$ such that $\<v\> = \<v'\>$.

(1) $\d(\varsigma(v)) = \d(v) - 1$ and $h(\varsigma(v)) = h(v) - 1$ if $h(v) \ge 1$;

(2) $h(\varsigma(v)) = 0$ if $h(v) = 0$;

(3) $v = v'$ if $h(v) = h(v') = 0$;

(4) $\d(v) \le \d(v')$ if and only if $h(v) \le h(v')$;

(5) $\d(\varsigma(v)) \le \d(\varsigma(v'))$ if $\d(v) \le \d(v')$;
\end{lem}
\begin{proof}
Notice that $[v]$ consists of exactly $n$ elements. Then statements (1) and (2) follow by definition.

(3). Suppose $h(v) = h(v') = 0$, that is, $\max [v] - \min [v], \max [v'] - \min [v'] < 1$. Note that $v - v' \in \ZZ^n$ and $\<v - v'\> = 0$. If $v - v' \neq 0$, then there exist $1 \le i \neq j \le n$ such that $v(i) - v'(i), v'(j) - v(j) \in \ZZ_{\ge 1}$. Thus $$v(i) - v(j) = (v(i) - v'(i)) + (v'(i) - v'(j)) + (v(j)' - v(j)) > 1 - 1 + 1 = 1,$$ which is a contradiction.

(4). Suppose $h(v) \le h(v')$. By (1) we may replace $v, v'$ with $\varsigma^{h(v)}(v), \varsigma^{h(v)}(v')$ respectively so that $h(v) = 0$. In particular, $\d(v) < n-1$. To show $\d(v) \le \d(v')$ we may assume $h(v') \le \d(v') < n-1$. By (1) and (2) we have $$h(\varsigma^{h(v')}(v)) = h(\varsigma^{h(v')}(v')) = 0,$$ which means $\varsigma^{h(v')}(v) = \varsigma^{h(v')}(v')$ by (3). Write $$[\varsigma^{h(v')}(v)] = [\varsigma^{h(v')}(v')] = \{x_1, \dots, x_n\},$$ where $x_1 < x_2 < \cdots < x_n \in \RR$. As $h(v) = 0$, we deduce that $$[v] = \{x_{h(v')+1}, x_{h(v')+2}, \dots, x_n, x_1 + 1, x_2 + 1, \dots, x_{h(v')} + 1\},$$ which means that \begin{align*} \d(v) = \<\varsigma^{h(v')}(v')\> - n x_{h(v') + 1} + h(v') \le \<\varsigma^{h(v')}(v')\> - n x_1 + h(v') = \d(v'), \end{align*} where the equality holds if and only if $h(v') = 0 = h(v)$. This finishes the proof of (4).

(5). If $\d(v) \le \d(v')$, then $h(v) \le h(v')$ by (4). So $h(\varsigma(v)) \le h(\varsigma(v'))$ by (1) and (2), which means $\d(\varsigma(v)) \le \d(\varsigma(v'))$ by (4).
\end{proof}

Let $\bs: G^d(L) \to G^d(L)$ be the endomorphism defined in \S \ref{subsection_muticopy}. Then the induced linear map $\bs: (\RR^n)^N = Y_\RR^d \to Y_\RR^d = (\RR^n)^N$ is given by $$(v^1, v^2, \dots, v^N) \mapsto (\e^1 v^2, \dots \e^{N-1} v^N, \e^N v^1),$$ where $\e^k = p$ if $k \in d \ZZ$ and $\e^k = 1$ otherwise.

\begin{lem}\label{lemma_semimodule_sum_equal}Let $b$ and $\bmu$ as in Theorem \ref{thm A'}. Let $\bl, \eta_\bullet\in Y^d$ such that $C_{\bmu}^{\bl}(b)$ and $C_{\bmu}^{\eta_\bullet}(b)$ are both non-empty. Then $\<\l^k\>=\<\eta^k\>$ for any $1\leq k\leq N$.
\end{lem}
\begin{proof}By Proposition \ref{dim}, $-\bl+\tau_\bullet+\bw\bs(\bl)$ and $-\eta_{\bullet}+\tau_\bullet+\bw\bs(\eta_\bullet)$ are both conjugate to $\bmu$. In particular, for any $1\leq k\leq N$, \[\<\lambda^k\>-\<\eta^k\>=\e^k(\<\lambda^{k+1}\>-\<\eta^{k+1}\>).\] Therefore \[\<\lambda^k\>-\<\eta^k\>=(\prod_{i=1}^N\e^i)(\<\lambda^{k}\>-\<\eta^{k}\>)=p^f(\<\lambda^k\>-\<\eta^k\>),\] and the result follows.
\end{proof}

\begin{proof}[Proof of Theorem \ref{thm A'}]We first show each connected component $C$ contains some $C_\bmu^\bl(b)$ with $\dim C_\mu^\bl(b) =0$. Indeed, let $\bl \in Y^d$ be such that $C_\bmu^\bl(b) \cap C$ is non-empty and $d_{\bl} := \dim I^d u^\bl G^d(\co_L)/G^d(\co_L)$ is as small as possible. If $\dim C_\bmu^\bl(b) > 0$, then $C_\bmu^\bl(b)$ is irreducible of dimension $\ge 1$ by Proposition \ref{dim}. So $C_\bmu^\bl(b)$ is a closed (non-projective) subvariety of the affine space $I^d u^\bl G^d(\co_L)/G^d(\co_L)$, whose closure $\overline{C_\bmu^\bl(b)}$ must intersect $I^d u^\bchi G^d(\co_L)/G^d(\co_L)$ for some $\bchi \neq \bl$. So $C_\bmu^\bchi(\bb) \neq \emptyset$ and $d_\bchi < d_\bl$ since $$I^d u^\bchi G^d(\co_L)/G^d(\co_L) \subseteq \overline{I^d u^\bl G^d(\co_L)/G^d(\co_L)}.$$ This contradicts the choice of $\bl$. Therefore, $\dim C_\bmu^\bl(b) = 0$.

It remains to show the uniqueness of $\bl$. We need two Claims.

By assumption, $b = \dot \tw$ for some $\tw = u^\t w \in \tW$ such that the fixed point $e \in Y_\RR$ of $\tilde{w}\sigma$ lies in the fundamental alcove $\D$. Set $\be = (e, \dots, e) \in (Y_\RR)^d$ and $\btw = (1, \dots, 1, \tw) = u^{\bt} \bw$ with $\bt = (0, \dots, 0, \t)$ and $\bw = (1, \dots, 1, w)$. Let $\hbl = \bl - \be \in (\RR^n)^N$.

\emph{Claim 1}: $\hl^k = \begin{cases}\e^k w^k(\hl^{k+1}), &\text{ if } m^k = 0;\\ \varsigma(\e^k w^k(\hl^{k+1})), &\text{ if } m^k = 1.\end{cases}$

\emph{Claim 2}: there exists $1\leq k_0\leq N$ such that $h(\hl^{k_0}) = 0$.

Assume these two Claims for the moment. Suppose there exists another cocharacter $\bbeta \in Y^d$ such that $\dim C_\bmu^\bbeta(b) = 0$. Let $\hat \eta_\bullet = \bbeta - \be$. Then $\<\hl^k\> = \<\he^k\>$ for any $k$ by Lemma \ref{lemma_semimodule_sum_equal}. We may assume $\d(\hl^1) \le \d(\he^1)$, and it follows from Claim 1, Lemma \ref{delta} (4) and (5) that $\d(\hl^k) \le \d(\he^k)$ and $h(\hl^k) \le h(\he^k)$ for any $k$.  By Claim 2, there exists $k_0$ such that $h(\he^{k_0}) = 0$ and hence $h(\he^{k_0}) = 0 = h(\hl^{k_0})$. By Lemma \ref{delta} (3), we have $\hl^{k_0} = \he^{k_0}$, which implies by Claim 1 that $\hl^k = \he^k$ for any $k$ as desired. This proved the uniqueness of $\bl$.

Now we prove Claim 1 and 2. For Claim 1, as $\be = \tw_{\bullet}\bs(\be) = \bt + \bw\bs(\be)$, we have $ \bbl^\natural = -\hbl + \bw\bs(\hbl)$. Therefore we have $(\l^\natural)^k = -\hl^k + \e^k w^k(\hl^{k+1})$. As $\bmu$ is minuscule, it follows from Proposition \ref{dim} that $ \bbl^\natural$ is conjugate to $\bmu$. It's equivalent to say $(\l^\natural)^k$ and $\mu^k = m^k \o_1^\vee$ are conjugate under $\fs_n$. If $m^k = 0$, Claim 1 follows.  Now assume $m^k = 1$, then $(\l^\natural)^k$ and $\o_1^\vee$ are conjugate. So there exists $1 \le i_0 \le n$ such that $(\l^\natural)^k(i) = 1$ if $i = i_0$ and $(\l^\natural)^k(i) = 0$ otherwise. Equivalently,
\[\hl^k(i) =\begin{cases} \e^k w^k(\hl^{k+1})(i) &\text{ if } i\neq i_0\\ \e^k w^k(\hl^{k+1})(i)-1 &\text{ if } i= i_0.\end{cases}\]

 Let $x = (w^k(\hl^{k+1}))(i_0) \in [\hl^{k+1}]$. Notice that $[\hl^k]$ consists of exactly $n$ elements. In order to prove Claim 1, it suffices to show that $x = \max[\hl^{k+1}]$. If it's not the case, then $(w^k(\hl^{k+1}))(i_1) = \max{[\hl^{k+1}]} > x$ for some $1 \le i_1 \neq i_0 \le n$. Thus $\hl^k(i_1) = \e^k \max[\hl^{k+1}]$ and $\hl^k(i_0) = \e^k x - 1$. Let $\a = \a_{i_1, i_0} \in \Phi$. We will show  $\a \in R(\bl)$ which contradicts the fact that $\dim C_\bmu^\bl(b) = |R(\bl)| = 0$ by Proposition \ref{dim}.
 Indeed, we have $$\<\a, (\l^\natural)^k\> = (\l^\natural)^k(i_1) - (\l^\natural)^k(i_0) = -1,$$ and moreover, $$\<\a, \l^k\> = \<\a, \hl^k\> + \<\a, e^k\> = \e^k \max[\hl^{k+1}] - (\e^k x - 1) + \<\a, e^k\> > 1 + \<\a, e^k\>.$$ As $e^k$ lies in the fundamental alcove, we have $-1 < \<\a, e^k\> < 1$. Moreover $\<\a, \l^k\> \ge 2$ if $\a > 0$ and $\<\a, \l^k\> \ge 1$ if $\a < 0$. So $(\l^k)_\a \ge 1$ and hence $\a \in R(\bl)$. This finishes the proof of Claim 1.

For Claim 2, suppose that $h(\hl^{k+1}) \ge 1$ for any $k$. It follows from Claim 1 that $\min[\hl^k] = \e^k \min[\hl^{k+1}]$ for any $k$. In particular, $$\min[\hl^k] = (\prod_{l=1}^N \e^l) \min[\hl^k] = p^f \min[\hl^k],$$ which means that $\min[\hl^k] = 0$. This is impossible since $b$ is simple and $[e^k] \cap \ZZ = \emptyset$ by Proposition \ref{simple}. This proves Claim 2.

\end{proof}

\subsection{} Now we show the second connectedness result for the totally ramified case with $n=3$ (i.e. $G=\GL_3$).
\begin{thm}\label{thm B}Suppose $G=\mathrm{GL}_3$ and $b$ is simple, then $C_{\mu}(b)$ is geometrically connected.
\end{thm}
\begin{proof}By Proposition \ref{Caruso} and \ref{simple}, we may assume $b=au^\tau\omega$ satisfies Lemma \ref{iota} with $a\in G(\bar\FF_p)$ central, $\tau\in Y$ and $w\in \fs_3$ of order 3. Obviously $C_{\mu}(b)= C_{\mu}(u^\tau\omega)$, so we may assume $a=1$. Let
\[S:=\{\lambda\in Y| \lambda^\natural\leq \mu\}\]
where $\lambda^{\natural}:=-\lambda+\tau+w\sigma(\lambda)$ as before. By Proposition \ref{semi-module}, it suffices to connect the points
$\{u^{\lambda}\in C_{\mu}(b)(\bar\FF_p)| \lambda\in S\}$ inside $C_{\mu}(b)$. The result follows from the following two claims.

\emph{Claim 1:} Suppose $\lambda'=\lambda-\alpha^{\vee}$ with $\lambda,\lambda'\in S$ and $\alpha^{\vee}$ is a coroot, then $u^{\lambda}$ and $u^{\lambda'}$ are in the same geometrically connected component of $C_{\mu}(b)$.

\emph{Claim 2:} For any different $\lambda, \lambda'\in S$, there exists a chain of elements $\lambda_0,\cdots\lambda_r\in S$ for some $r\in\NN$ such that $\lambda=\lambda_0$, $\lambda'=\lambda_r$ and $\lambda_{i+1}-\lambda_i$ is a coroot for any $0\leq i\leq r-1$.

We first prove Claim 1. Let $U_\alpha\subset G$ be the root subgroup of $G$ corresponding to $\alpha$. We fix an isomorphism $\GG_a\simeq U_{\alpha}$. For any $\bar\FF_p$-algebra $R$, we write $U_{\alpha}(x)\in U_{\alpha}(R)$ for the image of $x\in R$ via  $\GG_a\simeq U_{\alpha}$.

For any $x\in \bar \FF_p$, let $g(x):=u^{\lambda}U_\alpha(u^{-1}x)$.

We verify that $g(x)K\in C_{\mu}(b)(\bar\FF_p)$. Note that

\[g(x)^{-1}b\sigma g(x)=U_{\alpha}(-u^{-1}x)u^{\lambda^{\natural}}U_{w\alpha}(u^{-p}\alpha).\]

By a $\mathrm{SL}_2$-computation, we get

\[U_{\alpha}(u^{-m}x)\in U_{-\alpha}(u^mx^{-1})u^{-m\alpha^{\vee}}G(\co_L)\] for any positive integer $m$. So $g(x)^{-1}b\sigma g(x)\in \overline{ G(\co_L) u^\mu G(\co_L)}$ if
\[\lambda^{\natural}-pw\alpha^{\vee}\leq\mu, \lambda^{\natural}+\alpha^{\vee}\leq \mu \text{ and } \lambda'^{\natural}=\lambda^{\natural}+\alpha^{\vee}-pw\alpha^{\vee}\leq\mu.\]The first two conditions are implied by the third one as $\langle\alpha, w\alpha^{\vee}\rangle=-1$, and the third condition is automatic as $\lambda'\in S$.

Therefore, $g(x)$ defines an affine line $g:\AA^1\rightarrow C_{\mu}(b)$. As $C_{\mu}(b)$ is projective, $g$ extends uniquely to $g:\PP^1\rightarrow C_{\mu}(b)$. It's easy to compute $g(\infty)=u^{\lambda'}$. Hence $u^{\lambda}$ and $u^{\lambda'}$ are in the same connected component.

Now it remains to prove Claim 2. Let $\lambda, \lambda'\in S$. Then $\lambda-\lambda'$ is in the coroot lattice by Lemma \ref{lemma_semimodule_sum_equal}.

Note that $\alpha_1^\vee+w\alpha_1^{\vee}+w^2\alpha_1^{\vee}=0$ for any coroot $\alpha_1^\vee$.  It follows that we can always find a coroot $\alpha_1^{\vee}$ such that

\[\lambda'-\lambda=n_1\alpha_1^{\vee}+n_2w\alpha_1^{\vee}\]
with $n_1=\mathrm{max}\{|n_1|, |n_2|, |n_1-n_2|\}$. In particular $n_1\geq n_2\geq 0$. We will prove by induction on $n_1$. Let $\alpha_i^{\vee}:=w^{i-1}\alpha_1^{\vee}$ for $i\in\NN$.

If $n_2=n_1$ or $0$, then $\lambda'-\lambda=n_1\alpha^{\vee}$ is a multiple of the coroot $\alpha^{\vee}$ where  $\alpha_1^{\vee}$ equals to $-\alpha_3^{\vee}$ (resp. $\alpha_1^{\vee}$) if $n_2=n_1$ (resp. $n_2=0$). Then $\lambda+n\alpha^{\vee}\in S$ for $0\leq n\leq n_1$. We are done.

Now we may assume $n_1>n_2>0$.

\emph{Claim 3: } $\lambda+\alpha_1^{\vee}\in S$ or $\lambda-\alpha_3^{\vee}\in S$.

Suppose Claim 3 holds for the moment. Notice that $-\alpha_3^{\vee}=\alpha_1^{\vee}+\alpha_2^{\vee}$, we can apply induction hypothesis to the pair $(\lambda-\alpha_3^{\vee}, \lambda')$ if $\lambda-\alpha_3^{\vee}\in S$ or to the pair $(\lambda+\alpha_1^{\vee}, \lambda')$ if $\lambda+\alpha_1^{\vee}\in S$. This proves Claim 1.

Now it remains to prove Claim 3. Suppose Claim 3 does not hold. Without loose of generality, we may assume $\alpha_1=(1, -1, 0)$ and $\alpha_2=(0, 1, -1)$. Then

\[\begin{split}\lambda'^{\natural}&=\lambda^{\natural}-(n_1+pn_2)\alpha_1^{\vee}+(pn_1-(p+1)n_2)\alpha_2^{\vee}
\\ &=\lambda^{\natural}+(-n_1-pn_2, (p+1)n_1-n_2, -pn_1+(p+1)n_2).\end{split}\]
In particular,
\[\begin{split}(\lambda+\alpha_1^{\vee})^\natural&=\lambda^{\natural}+(-1, p+1, -p)
\\ (\lambda-\alpha_3^{\vee})^{\natural} &=\lambda^{\natural}+(-1-p, p, 1).\end{split}\]
Write $\lambda^{\natural}=(a_1, a_2, a_3)$. As  $\lambda'^{\natural}\leq\mu$ and $\lambda^{\natural}\leq \mu$,  we deduce that $(\lambda+\alpha_1^{\vee})^{\natural}\nleq \mu$ is equivalent to $a_3-p<\min[\mu]$ and $(\lambda-\alpha_3^{\vee})^{\natural}\nleq\mu$ is equivalent to $a_3+1>\max[\mu]$. Therefore, \[\max[\mu]-\min[\mu]<1+p.\] On the other hand, the fact that $\lambda'^{\natural}\leq\mu$ and $\lambda^{\natural}\leq \mu$ implies
\[\max[\mu]-\min[\mu]\geq n_1+pn_2\geq p+1,\]which is impossible.
 \end{proof}

\subsection{} \label{counterexmaples} In general, $C_{\mu}(b)$ is not connected. We give two counter-examples.
\begin{prop}
(a) Let $G=\GL_4$, $b=u^{(2,0,2,0)}(1243)\in \mathrm{GL}_n(\FF_p)$ and $\mu=(2p-1, p, p, 1)$. Then \[C_{\mu}(b)(\FF_p)=C_{\mu}(b)(\bar\FF_p)=\{u^{(2,1,1,0)}, u^{(1,1,1,1)}\}.\]

(b) Let $G=\mathrm{Res}_{k|\FF_p}\mathrm{GL}_3$ with $[k: \FF_p]=2$. Choose $\FF$ containing $k$. Then the group $G_{|\FF}\simeq \mathrm{GL}_3\times \mathrm{GL}_3$. Let $b=(u^{(2,0,1)}(123), u^{(0,0,1)})\in G(\FF)$ and $\mu=((p+1, 0,0), (p,p, 0))$, then \[C_{\mu}(b)(\FF)=C_{\mu}(b)(\bar\FF_p)=\{u^{\chi}, u^{\chi'}\},\] where $\chi=((1,0,1), (0,0,1))$ and $\chi'=((1,1,0), (1, 0, 0))$.
\end{prop}

\begin{proof}
For $\l \in Y$ we set $D(\l) = \{\a \in \Phi; \l_\a \ge 0, \<\a, \l^\natural\> \le -1\}$.

\emph{Claim:} if $\l^\natural$ is conjugate to $\mu$, and $\l_\a = 0$ for $\a \in D(\l)$, then $C_\mu^\l(b) = \{u^\l\}$.

Indeed, as $\l^\natural$ is conjugate to $\mu$, we have $$U_\l^+ u^{\l^\natural} \cap \overline{G(\co_L) u^\mu G(\co_L)}= U_\l^+(\co_L) u^{\l^\natural} U_\l^+(\co_L) = U_\l^+(\co_L) (\prod_{\a \in D(\l)} U_\a(u^{\<\a, \l^\natural\>} \co_L)) u^{\l^\natural}.$$ For $\a \in D$ we have $\l_\a = 0$ by assumption and hence \[I \cap  U_\a(u^{\<\a, \l\> + \<\a, \l^\natural\>} \co_L) = u^\l U_\a(\co_L) u^{-\l}.\] Thus \begin{align*} &\quad\ (I \cap U_\l^+) \cap u^\l \overline{G(\co_L) u^\mu G(\co_L)} u^{-\l^\dag} \\ & = u^\l U_\l^+(\co_L) u^{-\l} \prod_{\a \in D(\l)} I \cap  U_\a(u^{\<\a, \l\> + \<\a, \l^\natural\>} \co_L) \\ &= u^\l U_\l^+(\co_L) u^{-\l},\end{align*} which means $C_\mu^\l(b) = \{u^\l\}$ by Proposition \ref{semi-module} (b). This concludes the Claim.

In the case (a), one checks directly that $C_\mu^\l(b) \neq \emptyset$, or equivalently $\l^\natural \leq \mu$, if and only if $\l = (1, 1, 1, 1)$ or $\l = (2, 1, 1, 0)$. In the former case, we have $I u^\l G(\co_L)/G(\co_L) = u^\l G(\co_L)/G(\co_L)$ and $C_\mu^\l = \{u^\l\}$ since $\l$ is central. In the latter case, one checks that $\l$ satisfies the conditions of the Claim and it follows that $C_\mu^\l(b) = \{u^\l\}$ as desired.

In the case (b), one checks directly that $C_\mu^\l(b) \neq \emptyset$ if and only if $\l = \chi'$ or $\l = \chi$. If $\l = \chi'$, then $\l$ is dominant and minuscule, which means $I u^\l G(\co_L) / G(\co_L) = u^\l G(\co_L) / G(\co_L)$ and $C_\mu^\l(b) = \{u^\l\}$. If $\l = \chi$, then the conditions of the Claim are satisfied and hence $C_\mu^\l(b) = \{u^\l\}$ as desired.
\end{proof}

\section{Application to deformation spaces}\label{section_application}
In this section, we assume $p>2$. Suppose $\bar\rho: \Gamma_K\rightarrow \mathrm{GL}_n(\FF)$ is absolutely irreducible and flat. Here flat means that $\bar\rho$ comes from a finite flat group scheme over $\mathcal{O}_K$. Let $R^{fl}$ be the flat deformation ring of $\bar\rho$ in the sense of Ramakrishna (\cite{Ra}). It's a complete local noetherian $W(\FF)$-algebra.

We consider the conjugacy class $\{\nu\}$ of a minuscule cocharacter \[\nu: \GG_{m, \bar{\QQ}_p}\rightarrow (\mathrm{Res}_{K|\QQ_p}\mathrm{GL}_n)_{\bar{\QQ}_p}.\]  Let $T_n$ be the maximal torus of $\mathrm{GL}_n$ consisting of diagonal matrices. We may assume that $\nu$ is dominant and has image in $\mathrm{Res}_{K|\QQ_p}T_{n,\bar\QQ_p}$.  Write $\nu=(\nu_{\delta})_{\delta\in\mathrm{Hom}(K, \bar\QQ_p)}$ with \[\nu_{\delta}=(\underbrace{1,\cdots,1}_{v_{\delta}}, 0,\cdots, 0)\in X_*(T_n)=\ZZ^n.\]

We deonote by $R^{fl,\nu}$ for the quotient of $R^{fl}$ corresponding to the deformations $\xi: G_K\rightarrow \mathrm{GL}_n(\mathcal{O}_E)$ with $E$ an extension of $W(\FF)[\frac{1}{p}]$ that contains the reflex field of $\{\nu\}$, such that Hodge-Tate weights given by $\nu$, i.e., for any $a\in K$,
\[\det_E(a|D_{cris}(\xi)_K/\mathrm{Fil}^0D_{cris}(\xi)_K)=\prod_{\delta\in \Hom (K, \bar\QQ_p)} \delta(a)^{v_{\delta}}.\]

The cocharcter $\nu$ comes from a cocharacter over $\bar\ZZ_p$ that we still denote by $\nu$:
\[\nu: \GG_{m, \bar{\ZZ}_p}\rightarrow (\mathrm{Res}_{K|\ZZ_p}T_n)_{\bar{\ZZ}_p}.\] Therefore, we obtain a cocharacter
\[\mu(\nu): \GG_{m,\bar\FF_p}\stackrel{\nu\times_{\bar\ZZ_p}\bar\FF_p}{\longrightarrow}(\mathrm{Res}_{\mathcal{O}_K|\ZZ_p}(T_n))\times_{\bar\ZZ_p}\bar\FF_p\rightarrow \mathrm{Res}_{k|\FF_p}(T_n)_{\bar\FF_p},\]where the second arrow is the natural map induced from $\mathcal{O}_K\otimes_{\ZZ_p} \bar\FF_p\rightarrow k\otimes_{\FF_p}\bar\FF_p$. More concretely, write $\mu(\nu)=(\mu(\nu)_{\tau})_{\tau\in\mathrm{Hom}(k, \bar\FF_p)}$ with $\mu(\nu)_{\tau}\in X_*(T_n)=\ZZ^n$. Then for any $\tau\in\mathrm{Hom}(k, \bar\FF_p)$,
\[\mu(\nu)_{\tau}=\sum_{\delta\in \mathrm{Hom}(K, \bar\QQ_p)\atop \bar\delta=\tau}\nu_{\delta},\]where $\bar\delta$ is the embedding of the residue fields induced from $\delta$.

\begin{cor}\label{coro_application_deformation}
The scheme $\mathrm{Spec}(R^{\mathrm{fl},\nu}[\frac{1}{p}])$ is connected if one of the following two conditions holds:
\begin{enumerate}
\item $\mu(\nu)=(\mu(\nu)_{\tau})_{\tau}$ with $\mu(\nu)_{\tau}=(\mu(\nu)_{\tau,1},\cdots,\mu(\nu)_{\tau,n})$ such that $\mu(\nu)_{\tau,1}\geq\mu(\nu)_{\tau,2}=\cdots=\mu(\nu)_{\tau,n}$ for all $\tau\in\mathrm{Hom}(k,\bar\FF_p)$;
\item $K$ is totally ramified and $n=3$.
\end{enumerate}
\end{cor}
\begin{proof}Suppose the restriction to $G_{K_{\infty}}$ of the Tate-Twist $\bar\rho(-1)$ corresponding to an absolutely simple $\varphi$-module
\[((k\otimes_{\FF_p}\FF)^n, b\varphi)\]
of rank $n$ over $k\otimes_{\FF_p}\FF$, where $b\in \mathrm{Res}_{k|\FF_p}\GL_n(\FF)$.
By \cite{K09} 2.4.10, the connected components of $\mathrm{Spec}(R^{\mathrm{fl},\nu}[\frac{1}{p}])$ is in bijection with $C_{\mu(\nu)}(b)$. Then the result follows from Theorems \ref{thm A} and \ref{thm B}.
\end{proof}

\end{document}